\documentclass[11pt]{article}
\usepackage{fullpage} 

\usepackage{ucs}
\usepackage{amssymb}
\usepackage{amsthm}
\usepackage{amsmath}
\usepackage{latexsym}
\usepackage[cp1251]{inputenc}
\usepackage[english]{babel}
\usepackage{graphicx}
\usepackage{wrapfig}
\usepackage{caption}
\usepackage{subcaption}
\usepackage{txfonts}
\usepackage{lmodern}
\usepackage{mathrsfs}
\usepackage{eufrak}
\usepackage{algorithm}
\usepackage{dsfont}
\usepackage{enumerate}
\usepackage{hyperref}
\usepackage{multicol}
\usepackage{csquotes}

\title{On the number of edges in a graph with no $(k+1)$-connected subgraphs}
\date{}
\author{
Anton~Bernshteyn~\thanks{Department of Mathematics, University of Illinois at Urbana--Champaign, IL, USA, bernsht2@illinois.edu. Research of this author is supported by the Illinois Distinguished Fellowship.}
\and Alexandr Kostochka \thanks{Department of Mathematics, University of Illinois at Urbana--Champaign, IL, USA and
Sobolev Institute of Mathematics, Novosibirsk 630090, Russia, kostochk@math.uiuc.edu. Research of this author is supported in part by NSF grant
 DMS-1266016 and by grant 15-01-05867 of the Russian Foundation for Basic Research.}}

\newtheorem{theo}{Theorem}

\newtheorem{lemma}[theo]{Lemma}

\newtheorem{conj}[theo]{Conjecture}

\theoremstyle{definition}

\theoremstyle{remark}

\newenvironment{lemmabis}[1]
  {\renewcommand{\thetheo}{\ref{#1}$'$}%
   \addtocounter{theo}{-1}%
   \begin{lemma}}
  {\end{lemma}}
	
	\newenvironment{conjbis}[1]
  {\renewcommand{\thetheo}{\ref{#1}$'$}%
   \addtocounter{theo}{-1}%
   \begin{conj}}
  {\end{conj}}

\begin{document}
	\maketitle
	
\begin{abstract}
Mader proved that for $k\geq 2$ and $n\geq 2k$, every $n$-vertex graph with no $(k+1)$-connected subgraphs has at most
$(1+\frac{1}{\sqrt{2}})k(n-k)$ edges.
He also conjectured that for $n$ large with respect to $k$, every such graph has at most $\frac{3}{2}\left(k - \frac{1}{3}\right)(n-k)$
edges. Yuster improved Mader's upper bound to $\frac{193}{120}k(n-k)$ for $n\geq\frac{9k}{4}$. In this note, we make the next step towards Mader's
Conjecture: we improve Yuster's bound to $\frac{19}{12}k(n-k)$ for $n\geq\frac{5k}{2}$.
\end{abstract}

	\section{Introduction}
	
	All graphs considered here are finite, undirected, and simple. For a graph $G$, $V(G)$ and $E(G)$ denote 
its vertex set and edge set respectively. If $U \subseteq V(G)$, then $G[U]$ denotes the induced subgraph of $G$ whose vertex set is $U$, and $G - U \coloneqq G[V(G) \setminus U]$. 
For $v \in V(G)$, $N(v) \coloneqq \{u \in V(G)\,:\,uv \in E(G)\}$ denotes the neighborhood of $v$ in $G$.
	
	Let $k \in \mathbb{N}$. Recall that a graph $G$ is \emph{$(k+1)$-connected} if, for every set $S \subset V(G)$ of size $k$, 
the graph $G[V(G)\setminus S]$ is connected and contains at least two vertices (so $|V(G)| \geq k+2$). 
Mader \cite{Mader2} posed the following  question:
\begin{displayquote}
	What is the maximum possible number of edges in an $n$-vertex graph that does not contain a $(k+1)$-connected subgraph?
\end{displayquote}
 It is easy to see that for $k = 1$ the answer is $n-1$: every tree on $n$ vertices contains $n-1$ edges and no $2$-connected subgraphs, 
whereas every graph on $n$ vertices with at least $n$ edges  contains a cycle, and cycles are $2$-connected.
 Thus for the rest of the note we will assume $k \geq 2$.
	
	The following construction due to Mader \cite{Mader1} gives an example of a graph with no $(k+1)$-connected subgraphs and a large number of edges. 
Fix $k$ and $n$, and suppose that $n = kq + r$, where $1 \leq r \leq k$. The graph $G_{n, k}$ has vertex set $\bigcup_{i = 0}^q V_i$, 
where the sets $V_0$, \ldots, $V_q$ are pairwise disjoint and satisfy the following contitions.
	\begin{enumerate}
		\item $|V_0| = \ldots = |V_{q-1}| = k$, while $|V_q| = r$.
		\item $V_0$ is an independent set in $G_{n, k}$.
		\item For $1 \leq i \leq q$, $V_i$ is a clique in $G_{n,k}$.
		\item Every vertex in $V_0$ is adjacent to every vertex in $\bigcup_{i=1}^q V_i$.
		\item  $G_{n,k}$ has no other edges.
	\end{enumerate}
	Note that $V_0$ is a separating set of size $k$ and every component of $G_{n,k}-V_0$ has at most $k$ vertices. 
	It follows that
 $G_{n,k}$ has no $(k+1)$-connected subgraphs. A direct calculation shows that $G_{n, k}$ has at most 
$\frac{3}{2}\left(k - \frac{1}{3}\right)(n-k)$ edges, where the equality holds if $n$ is a multiple of $k$. 
Mader \cite{Mader1} conjectured that this example is, in fact, best possible.
	
	\begin{conj}[Mader \cite{Mader1}]\label{conj:Mader}
		Let $k \geq 2$. Then for $n$ sufficiently large, the number of edges in an $n$-vertex graph without a $(k+1)$-connected subgraph cannot exceed
 $\frac{3}{2}\left(k - \frac{1}{3}\right)(n-k)$.
	\end{conj}
	
	Mader himself proved Conjecture \ref{conj:Mader} for $k\leq 6$. Moreover, he showed that for all $k$, the weaker version of the conjecture,
 where the coefficient $\frac{3}{2}$ is replaced by $1+\frac{1}{\sqrt{2}}$, holds. 
Yuster~\cite{Yuster} improved this result by showing that the coefficient can be taken to be $\frac{193}{120}$.

	\begin{theo}[Yuster~\cite{Yuster}]\label{theo:main}
		Let $k \geq 2$ and $n\geq \frac{9k}{4}$. Then every $n$-vertex graph $G$ with  $|E(G)| > \frac{193}{120}k(n-k)$ contains a $(k+1)$-connected subgraph. 
	\end{theo}	

 Here we improve Yuster's bound, obtaining the value $\frac{19}{12}$ for the coefficient.

	It turns out that for this problem, computations work out nicer if we ``normalize'' vertex and edge counts by assigning a weight $\frac{1}{k}$ 
to each vertex and a weight $\frac{1}{k^2}$ to each edge in a graph. Using this terminology, we can restate Conjecture \ref{conj:Mader} in the following way.
	
	\begin{conjbis}{conj:Mader}\label{conj:Mader1}
		Let $k \geq 2$. Then for $\gamma$ sufficiently large, every graph $G$ with $\frac{1}{k}|V(G)| = \gamma$ and $\frac{1}{k^2}|E(G)| > \frac{3}{2}(\gamma-1)$ contains a $(k+1)$-connected subgraph.
	\end{conjbis}

Our main result in these terms is as follows.

	\begin{theo}\label{theo:main}
		Let $k \geq 2$. Then every graph $G$ with $\frac{1}{k}|V(G)| = \gamma \geq \frac{5}{2}$ and $\frac{1}{k^2}|E(G)| > \frac{19}{12}(\gamma-1)$ contains a $(k+1)$-connected subgraph. 
	\end{theo}	

We follow the ideas of Mader and Yuster: Use induction on the number of vertices for graphs with at least $\frac{5}{2}k$ vertices.
The hardest part is to prove the case when after deleting a separating set of size $k$, exactly one of the components of the
remaining graph has fewer than $\frac{3}{2}k$ vertices, since the induction assuption does not hold for $n<\frac{5}{2}k$. 
New ideas in the proof are in Lemmas~\ref{lemma:core},~\ref{lemma:smaller}, and~\ref{lemma:greater} below.
	
	\section{Proof of Theorem \ref{theo:main}}
	
We want to derive a linear in $(n-k)$ bound on the number of edges in a graph that does not contain $(k+1)$-connected subgraphs. 
But the bound becomes linear only for graphs with  large number of vertices; while for small graphs the dependency is quadratic in $n-k$. 
The main difficulties we encounter are around the transition between the quadratic and linear regimes. 
To deal with small $n$, we use the following lemma due to Matula \cite{Matula}, whose bound
is asymptotically exact for $n<2k$.
	
	\begin{lemma}[Matula \cite{Matula}]\label{lemma:matula}
		Let $k \geq 2$. Then every graph $G$ with $|V(G)| = n \geq k+1$ and $|E(G)| > {n \choose 2} - \frac{1}{3}((n-k)^2 - 1)$ contains a $(k+1)$-connected subgraph.
	\end{lemma}
	
	We will use the following ``normalized'' version of this lemma.
	
	\begin{lemmabis}{lemma:matula}\label{lemma:matula1}
		Let $k \geq 2$. Then every graph $G$ with $\frac{1}{k}|V(G)| = \gamma > 1$, and 
\begin{equation}\label{a12}
\frac{1}{k^2}|E(G)| > \frac{1}{6}\left(\gamma^2+4\gamma-2\right)
\end{equation}
 contains a $(k+1)$-connected subgraph.
	\end{lemmabis}
	
	\begin{proof}
		Indeed,~\eqref{a12} yields
		\begin{align*}
			|E(G)| &\,>\, \frac{k^2}{6}\left(\gamma^2+4\gamma-2\right) \\
			       &\,=\, {\gamma k \choose 2} - \frac{1}{3}((\gamma k-k)^2 - 1) +\frac{\gamma k}{2} - \frac{1}{3} \\
						 &\,>\, {\gamma k \choose 2} - \frac{1}{3}((\gamma k-k)^2 - 1),
		\end{align*}
		and we are done by original Matula's lemma.
	\end{proof}
	
	From now on, fix a graph $G$ with $\frac{1}{k}|V(G)| = \gamma \geq \frac{5}{2}$ and $\frac{1}{k^2}|E(G)| > \frac{19}{12}(\gamma-1)$, 
and suppose for contradiction that $G$ does not contain a $(k+1)$-connected subgraph. 
Choose $G$ to have the least possible number of vertices (so we can apply induction hypothesis for subgraphs of $G$). 
Since $G$ itself is not $(k+1)$-connected, it contains a separating set $S \subset V(G)$ of size $k$. Let $A\subset V(G) \setminus S$ 
be such that $G[A]$ is a smallest connected component of $G - S$, and let $B \coloneqq V(G) \setminus (S \cup A)$. 
Let $\alpha \coloneqq \frac{1}{k}|A|$ and $\beta \coloneqq \frac{1}{k}|B|$.
	
We start by showing that the graph $G$ cannot be too small, using Matula's Lemma.
	
	\begin{lemma}
		$\gamma > 3$.
	\end{lemma}
	\begin{proof}
		Suppose that $\gamma \leq 3$. Then, by Lemma~\ref{lemma:matula}$'$,
		\begin{equation}\label{eq:base}
			0 \,\leq\, \frac{1}{k^2}|E(G)| - \frac{19}{12}(\gamma-1) \,\leq\, \frac{1}{6}\left(\gamma^2+4\gamma-2\right) - \frac{19}{12}(\gamma-1) \,=\, \frac{1}{12}(2\gamma^2 - 11\gamma + 15).
		\end{equation}
The function $g(\gamma)=2\gamma^2 - 11\gamma + 15$ on the right-hand side of (\ref{eq:base}) is convex in $\gamma$. 
Hence it is maximized on the boundary of the interval $[\frac{5}{2}; 3]$. But it is easy to check that $g(\frac{5}{2})=g(3)=0$,
 hence it is nonpositive on the whole interval. Therefore, $\gamma > 3$.
	\end{proof}
	
		All the edges in $G$ either belong to the graph $G[S \cup B]$, or are incident 
to the vertices in $A$. The number of edges in $G[S \cup B]$ can be bounded either using Matula's lemma 
(which is efficient for $\beta \leq \frac{3}{2}$) or using the induction hypothesis 
(which can be applied if $\beta > \frac{3}{2}$). Hence the difficulty is  in bounding the number of edges incident to the vertices in $A$. 
	
	The first step is to show that $A$ cannot be too large, because otherwise we can use  induction.
	
	\begin{lemma}
		$\alpha < \frac{3}{2}$.
	\end{lemma}
	\begin{proof}
		If $\alpha \geq \frac{3}{2}$, then we can apply the induction hypothesis both for $G[S \cup A]$ and for $G [S \cup B]$, and thus obtain
	$$
		\frac{1}{k^2}|E(G)| \,\leq\, \frac{19}{12} \alpha + \frac{19}{12} \beta \,=\, \frac{19}{12}(\alpha+\beta) \,=\, \frac{19}{12}(\gamma-1).
	$$
	\end{proof}
	
	The next lemma shows that $A$ cannot be too small either, 
since otherwise the total number of edges between the vertices in $A$ and the vertices in $S \cup A$ is small.
	
	\begin{lemma}\label{lemma:nottoosmall}
		$\alpha > 1$.
	\end{lemma}
	\begin{proof}
		Suppose that $\alpha \leq 1$. Then $\beta > 1$, since $\alpha+\beta+1 = \gamma > 3$. 
If $\beta \geq \frac{3}{2}$, then using the induction hypothesis for $G[S \cup B]$, we get
$$
	\frac{1}{k^2}|E(G)| \,\leq\, \frac{1}{2}\alpha^2 + \alpha + \frac{19}{12}\beta \,\leq\, \frac{3}{2}\alpha + \frac{19}{12}\beta \,<\, 
\frac{19}{12}(\alpha+\beta) \,=\, \frac{19}{12}(\gamma-1). 
$$
Thus $\beta < \frac{3}{2}$. Therefore, $\alpha > \frac{1}{2}$. In this case, applying Lemma~\ref{lemma:matula}$'$ to $G[S \cup B]$ reduces the 
problem to proving the inequality
		$$
			\frac{1}{2}\alpha^2 + \alpha + \frac{1}{6}\left((\beta+1)^2+4(\beta+1)-2\right) \,\leq\, \frac{19}{12} (\alpha + \beta),
		$$
		which is equivalent to
		\begin{equation}\label{eq:aissmall}
			6\alpha^2 + 2\beta^2 -7\alpha - 7\beta + 6 \,\leq\, 0. 
		\end{equation}
		For $\alpha$ fixed, the left-hand side of (\ref{eq:aissmall}) is monotone decreasing in $\beta$ 
when $\beta < \frac{7}{4}$, so its maximum is attained at $\beta = 1$. Thus~\eqref{eq:aissmall} will hold if the function
$g_1(\alpha) =6\alpha^2  -7\alpha +1 $ is nonpositive.
Since $g_1(\alpha)$ is a convex function, 
	 its maximum on the interval $[\frac{1}{2}; 1]$ 
is attained at one of the boundary points. We have
		$$
			g_1\left(\frac{1}{2}\right)\,=\,6\cdot\left(\frac{1}{2}\right)^2  -7\cdot\frac{1}{2} +1\,=\,-1 \,<\, 0,
		\qquad
\mbox{	and}\qquad
			g_1(1)\,=\,6\cdot 1^2 - 7\cdot 1 + 1 \,=\, 0. \qedhere
		$$
	\end{proof}
	
	So we know that $1 < \alpha < \frac{3}{2}$. How can we bound the number of edges incident to the vertices in $A$?
The tricks of Lemma \ref{lemma:nottoosmall} 
  and of Lemma~\ref{lemma:matula1} are not sufficient here. The idea is to combine them by applying Lemma~\ref{lemma:matula1} only to the graph $G[A \cup (S \setminus S')]$, where $S'$ is a subset of $S$ with relatively few edges between  $A$ and $S'$. To obtain such set $S'$, 
we will use Lemma \ref{lemma:core} below, which asserts that there are many vertices in $S$ that have not too many neighbors in $A$.
	
	\begin{lemma}\label{lemma:core}
		Let $S_1 \coloneqq \{v \in S \,:\, \frac{1}{k}|N(v) \cap A| \leq \frac{1}{2}(\alpha+1)\}$. Then $\frac{1}{k}|S_1| > \frac{1}{3}$.
	\end{lemma}
	\begin{proof}
		Suppose that $\frac{1}{k}|S_1| = \sigma \leq \frac{1}{3}$. Let $G_1 \coloneqq G[A \cup (S \setminus S_1)]$. Since $G_1$ is not $(k+1)$-connected, it has a separating set $T \subset V(G_1)$ of size $k$. Let $X$ and $Y$ form a partition of $V(G_1) \setminus T$ and be separated by $T$ in $G_1$. Without loss of generality assume that $|X \cap A| \geq |Y \cap A|$. Then
		$$\frac{1}{k}|X\cap A| \,\geq\, \frac{1}{2}\cdot \frac{1}{k}|A \setminus T| \,\geq\, \frac{1}{2}(\alpha - 1).$$
		Hence if $v \in Y \cap S$, then
		$$\frac{1}{k}|N(v)\cap A|\,\leq\,\frac{1}{k}\left(|A| - |X \cap A|\right) \,\leq\, \frac{1}{2}(\alpha+1),$$
		which means that $v \in S_1$. But that is impossible, since $S_1 \cap V(G_1) = \emptyset$. Thus $Y \cap S = \emptyset$, i.e. $Y \subset A$. In particular, since $|X \cap A| \geq |Y \cap A| = |Y|$, we have $\frac{1}{k}|Y| \leq \frac{1}{2}\alpha$. Then
		$$\frac{1}{k}|V(G) \setminus Y| \,=\, \alpha+\beta+1 - \frac{1}{k}|Y| \,\geq\, \frac{1}{2}\alpha + \beta + 1 \,\geq\, \frac{5}{2},$$
		so the induction hypothesis holds for $G - Y$, and
		$$\frac{1}{k^2}|E(G-Y)| \,\leq\, \frac{19}{12} \left(\frac{1}{k}|V(G-Y)|-1\right).$$
		Hence we are done if
		$$\frac{1}{k^2}(|E(G)| - |E(G - Y)|) \,\leq\, \frac{19}{12} \cdot \frac{1}{k}|Y|,$$
		so assume that that is not the case. Let $\mu \coloneqq \frac{1}{k}|Y|$. Then
		$$\frac{1}{2}\mu^2 + \mu (1 + \sigma) \,>\, \frac{19}{12} \mu,$$
so
		$$\mu \,>\, \frac{7}{6} - 2\sigma.$$
	
	We consider  two cases.
	
	{\sc Case 1: $X \cap S \neq \emptyset$}. Let $v \in X \cap S$. Then $v$ has more than $k\cdot \frac{1}{2}(\alpha+1)$ neighbors in $A$, none of which belong to $Y$. Hence $\mu < \frac{1}{2}(\alpha-1)$, and so
	$$\frac{1}{2} (\alpha - 1) \,>\, \frac{7}{6} - 2\sigma.$$
	Therefore,
	$$\alpha \,>\, \frac{10}{3} - 4\sigma \,\geq\, \frac{10}{3} - 4\cdot\frac{1}{3}\,=\, 2;$$
	a contradiction.
	
	{\sc Case 2: $X \cap S = \emptyset$}. Then $S \setminus S_1 \subset T$, and the set $T \cap A$ separates $X$ and $Y$ in $G[A]$ and satisfies $\frac{1}{k}|T \cap A| = \frac{1}{k}(|T| - |T \cap S|) = 1 - (1-\sigma) = \sigma$. Note that since $|Y| \leq |X|$, we have
	$$\frac{7}{6} - 2\sigma \,<\, \mu \,\leq\, \frac{1}{2}(\alpha - \sigma),$$
	so
	$$\sigma \,>\, \frac{7}{9} - \frac{1}{3}\alpha \,>\, \frac{7}{9} - \frac{1}{3} \cdot \frac{3}{2} \,=\, \frac{5}{18}.$$
	
	Now observe that
	$$
		\frac{1}{k^2}|E(G[A])| \,\leq\, \frac{1}{2} \alpha^2 - \mu (\alpha - \sigma - \mu).
	$$
	Since $\frac{7}{6} - 2\sigma < \mu \leq \frac{1}{2}(\alpha-\sigma)$, the latter expression is less than
	$$
		\frac{1}{2} \alpha^2 - \left(\frac{7}{6} - 2\sigma\right)\cdot\left(\alpha + \sigma - \frac{7}{6}\right).
	$$
	Hence $\frac{1}{k^2}(|E(G)| - |E(G - A)|)$ is less than
	$$
		\frac{1}{2}(\alpha + 1) \sigma + \alpha (1 - \sigma) + \frac{1}{2} \alpha^2 - \left(\frac{7}{6} - 2\sigma\right)\cdot\left(\alpha + \sigma - \frac{7}{6}\right).
	$$
	
	{\sc Case 2.1. $\beta \leq \frac{3}{2}$.} Then, after adding Matula's estimate for the number of edges in $G[S \cup B]$ and subtracting $\frac{19}{12}(\alpha + \beta)$, it is enough to prove that the following quantity is nonpositive:
	\begin{align*}
		&\frac{1}{2}(\alpha + 1) \sigma + \alpha (1 - \sigma) + \frac{1}{2} \alpha^2 - \left(\frac{7}{6} - 2\sigma\right)\cdot\left(\alpha + \sigma - \frac{7}{6}\right) \\
		+ &\frac{1}{6}\left((\beta+1)^2+4(\beta+1)-2\right) - \frac{19}{12}(\alpha +\beta),
	\end{align*}
	which in equal to
	\begin{align*}
		\frac{1}{36}(18 \alpha^2+54 \alpha \sigma -63 \alpha+6 \beta^2-21 \beta+72 \sigma^2-108 \sigma+67).
	\end{align*}
	Note that for $\alpha$ and $\sigma$ fixed, the last expression is monotone decreasing in $\beta$ (recall that $\beta \leq \frac{3}{2}$, while the minimum is attained at $\beta = \frac{7}{4}$), so its maximum is attained when $\beta = \alpha$, where it turns into
	$$
		\varphi_1(\alpha,\sigma)\,=\,\frac{1}{36} (24 \alpha^2+54 \alpha \sigma-84 \alpha+72 \sigma^2-108 \sigma+67).
	$$
Since $\varphi_1(\alpha,\sigma)$  is convex in both $\alpha$ and $\sigma$, it attains its maximum at some point $(\alpha_0,\sigma_0)$, where
 $\alpha_0 \in \{1, \frac{3}{2}\}$ and $\sigma_0 \in \{\frac{5}{18}, \frac{1}{3}\}$. It remains to check the four possibilities:
 $$\varphi_1\left(1,\frac{5}{18}\right)\,=\,-\frac{11}{162} \,<\, 0,\quad \varphi_1\left(1,\frac{1}{3}\right)\,=\,-\frac{1}{12} \,<\, 0,$$
 $$\varphi_1\left(\frac{3}{2},\frac{5}{18}\right)\,=\,-\frac{125}{648} \,<\, 0,\quad\mbox{and}\quad \varphi_1\left(\frac{3}{2},\frac{1}{3}\right)\,=\,-\frac{1}{6} \,<\, 0.$$	
	
	{\sc Case 2.2. $\beta > \frac{3}{2}$.} Then, instead of using Matula's bound for $G[S \cup B]$, we can apply the induction hypothesis, so it is enough to prove that
	the function
	$$
	\varphi_2(\alpha,\sigma)\,=\,	\frac{1}{2}(\alpha + 1) \sigma + \alpha (1 - \sigma) + \frac{1}{2} \alpha^2 - \left(\frac{7}{6} - 2\sigma\right)\cdot\left(\alpha + \sigma - \frac{7}{6}\right) - \frac{19}{12}\alpha 
	$$
is nonpositive.
	Again, we only have to check the boundary values:
 $$\varphi_2\left(1,\frac{5}{18}\right)\,=\,-\frac{49}{324} \,<\, 0,\quad \varphi_2\left(1,\frac{1}{3}\right)\,=\,-\frac{1}{6} \,<\, 0,$$
 $$\varphi_2\left(\frac{3}{2},\frac{5}{18}\right)\,=\,-\frac{125}{648} \,<\, 0,\quad\mbox{and}\quad \varphi_2\left(\frac{3}{2},\frac{1}{3}\right)\,=\,-\frac{1}{6} \,<\, 0.$$
	This finishes the proof.
	\end{proof}
	
	Now we can simply try to  use as the set $S'$ the set $S_1$ itself. This choice indeed gives a good bound if $A$ is  large, as the next lemma shows.
	
	\begin{lemma}\label{lemma:smaller}
		$\alpha < \frac{4}{3}$.
	\end{lemma}
	\begin{proof}
		Suppose that $\alpha \geq \frac{4}{3}$. Recall that $\sigma = |S_1| > \frac{1}{3}$. Using Lemma~\ref{lemma:matula1} for $G[A \cup (S \setminus S_1)]$, we get that
		\begin{align*}
			&\frac{1}{k^2}(|E(G)| - |E(G - A)|) \,\\
			\leq\, &\frac{1}{6}\left((\alpha + 1 - \sigma)^2+4(\alpha + 1 - \sigma)-2\right) + \frac{1}{2}(\alpha+1)\sigma\,\\
			=\,&\frac{1}{6} (\alpha^2+\alpha \sigma+6 \alpha+\sigma^2-3 \sigma+3).
		\end{align*}
		
		{\sc Case 1: $\beta \leq \frac{3}{2}$.} Adding Matula's estimate for $G[B \cup S]$ and subtracting $\frac{19}{12}(\alpha+\beta)$, we get 
		\begin{align*}
			\frac{1}{k^2}|E(G)|&-\frac{19}{12}(\alpha+\beta)\,\\
			\leq\, &\frac{1}{6} (\alpha^2+\alpha \sigma+6 \alpha+\sigma^2-3 \sigma+3) + \frac{1}{6}\left((\beta+1)^2+4(\beta+1)-2\right) - \frac{19}{12}(\alpha+\beta) \\
			=\,& \frac{1}{12}(2 \alpha^2+2 \alpha \sigma-7 \alpha+2 \beta^2-7 \beta+2 \sigma^2-6 \sigma+12).
		\end{align*}
		Again, the maximum is attained when $\beta = \alpha$, so we should consider the expression
		$$
		\varphi_3(\alpha,\sigma)\,=\,\frac{1}{6} (2 \alpha^2+\alpha \sigma-7 \alpha+\sigma^2-3 \sigma+6).
		$$
It is convex in both $\alpha$ and $\sigma$, so again it is enough to check the boundary points:
$$\varphi_3\left(\frac{4}{3},\frac{1}{3}\right)\,=\,-\frac{1}{27} < 0,\quad \varphi_3\left(\frac{4}{3},1\right)\,=\,-\frac{2}{27} < 0,$$
$$\varphi_3\left(\frac{3}{2},\frac{1}{3}\right)\,=\,-\frac{7}{108} < 0,\quad\mbox{and}\quad \varphi_3\left(\frac{3}{2},1\right)\,=\,-\frac{1}{12} < 0.$$	
		
		{\sc Case 2: $\beta > \frac{3}{2}$.} Then, instead of using Matula's bound for $G[S \cup B]$, we can apply the induction hypothesis, so it is enough to prove that
the function
	$$
\varphi_4(\alpha,\sigma)\,=\,\frac{1}{6} (\alpha^2+\alpha \sigma+6 \alpha+\sigma^2-3 \sigma+3) - \frac{19}{12}\alpha \,=\, \frac{1}{12} (2 \alpha^2+2 \alpha \sigma-7 \alpha+2 \sigma^2-6 \sigma+6) 
		$$
is nonpositive. 
		The function is convex in both $\alpha$ and $\sigma$, so we check the boundary points:
$$\varphi_4\left(\frac{4}{3},\frac{1}{3}\right)\,=\,-\frac{1}{18} < 0,\quad \varphi_4\left(\frac{4}{3},1\right)\,=\,-\frac{5}{54} < 0,$$
$$\varphi_4\left(\frac{3}{2},\frac{1}{3}\right)\,=\,-\frac{7}{108} < 0,\quad\mbox{and}\quad \varphi_4\left(\frac{3}{2},1\right)\,=\,-\frac{1}{12} < 0.$$				
		This finishes the proof.
	\end{proof}
	
	The next lemma is the final piece of the jigsaw. It shows that if $A$ is small, we can still obtain the desired bound if we take the set $S'$ to be slightly bigger than $S_1$.
	
	\begin{lemma}\label{lemma:greater}
		$\alpha > \frac{4}{3}$.
	\end{lemma}
	\begin{proof}
		Suppose that $\alpha \leq \frac{4}{3}$. Then $1 - 2(\alpha-1) \geq \frac{1}{3}$. Let $S'$ be a subset of $S$ with $\frac{1}{k}|S'| = 1 - 2(\alpha-1)$ such that $\frac{1}{k}|S' \cap S_1| \geq \frac{1}{3}$. Observe that the normalized number of edges between $A$ and $S'$ is at most
		$$
			\frac{1}{k^2} |A|\cdot |S'| - \frac{1}{2}(\alpha - 1) \cdot \frac{1}{3}, 
		$$
		by the definition of $S_1$. Hence, using Lemma~\ref{lemma:matula1} for $G[A \cup (S \setminus S')]$, we get that
		\begin{align*}
			\frac{1}{k^2}(|E(G)| - |E(G - A)|) \,&\leq\, 
			\frac{1}{6}\left((3\alpha-2)^2+4(3\alpha-2)-2\right) + \alpha (3-2\alpha) - \frac{1}{6}(\alpha-1)\,\\
			&=\,\frac{1}{6} (-3 \alpha^2+17 \alpha-5).
		\end{align*}
		
		{\sc Case 1: $\beta \leq \frac{3}{2}$.} Adding Matula's estimate for $G[B \cup S]$ and subtracting $\frac{19}{12}(\alpha+\beta)$, we get 
		\begin{align*}
		\frac{1}{k^2}|E(G)|-\frac{19}{12}(\alpha+\beta)&\leq	\frac{1}{6} (-3 \alpha^2+17 \alpha-5) + \frac{1}{6}\left((\beta+1)^2+4(\beta+1)-2\right) - \frac{19}{12}(\alpha+\beta)\,\\
			&=\,\frac{1}{12} (-6 \alpha^2+15 \alpha+2 \beta^2-7 \beta-4).
		\end{align*}
		Since $\alpha \leq \beta \leq \frac{3}{2}$, the maximum is attained when $\beta = \alpha$, in which case the last expression turns into $-\frac{1}{3}(\alpha - 1)^2 \leq 0$.
		
		{\sc Case 2: $\beta > \frac{3}{2}$.} Then, instead of using Matula's bound for $G[S \cup B]$, we can apply the induction hypothesis, so it is enough to prove that
\begin{equation}\label{a131}
			\frac{1}{6} (-3 \alpha^2+17 \alpha-5) - \frac{19}{12}\alpha \,=\, -\frac{1}{12} (6 \alpha^2-15\alpha+10) \,\leq\,0.
\end{equation}
	Since the discriminant of the quadratic $6 \alpha^2-15\alpha+10$ is negative,~\eqref{a131} 
  holds for all $\alpha$, and we are done.
	\end{proof}
	
	Since Lemmas \ref{lemma:smaller} and \ref{lemma:greater}  contradict each other, we conclude that such graph $G$ does not exist. This  completes the proof of the theorem.

\end{document}